\documentclass[12pt]{amsart}
\usepackage[margin=1.15in]{geometry}
\usepackage{ url, xcolor, amsthm, amssymb, enumerate, graphicx, hyperref}
\usepackage{cleveref}

\newcommand{\Or}{\mathrm{O}}
\newcommand{\SO}{\mathrm{SO}}
\newcommand{\SL}{\mathrm{SL}}
\newcommand{\C}{\mathbb{C}}
\renewcommand{\H}{\mathbb{H}}
\newcommand{\N}{\mathbb{N}}
\newcommand{\Q}{\mathbb{Q}}
\newcommand{\Z}{\mathbb{Z}}
\newcommand{\R}{\mathbb{R}}

\newcommand{\GL}{\mathrm{GL}}
\newcommand{\Isom}{\mathrm{Isom}}
\newcommand{\disc}{\mathrm{disc}}
\newcommand{\rank}{\mathrm{rank}}
\newcommand{\Ram}{\mathrm{Ram}}
\newcommand{\Br}{\mathrm{Br}}
\newcommand{\Ok}{\mathcal{O}_k}
\newcommand{\kk}{k^{\times}/k^{\times 2}}

\theoremstyle{plain}
\newtheorem{theorem}{Theorem}[section]
\newtheorem{lemma}[theorem]{Lemma}
\newtheorem{corollary}[theorem]{Corollary}
\newtheorem{proposition}[theorem]{Proposition}

\newtheorem{definition}[theorem]{Definition}
\newtheorem{main}{Theorem}

\title[Salem numbers and arithmetic hyperbolic manifolds]{Salem numbers and commensurability classes of arithmetic hyperbolic manifolds}

\author[M. Chu]{Michelle Chu}
\address{School of Mathematics\\ University of Minnesota \\ 
127 Vincent Hall 206 Church St. SE\\ Minneapolis, MN 55455, USA.}
\email{mchu@umn.edu}
\thanks{The first author was partially supported by National Science Foundation grant DMS-2441034}

\author[P. G. P. Murillo]{Plinio G. P. Murillo}
\address{Instituto de Matemática e Estatística\\ Universidade Federal Fluminense \\
Rua Professor Marcos Waldemar de Freitas Reis,  s/n, Campus do Gragoatá, São Domingos, Niterói, RJ.
24210-201, Brazil.}
\email{pliniom@id.uff.br}
\thanks{The second author was partially supported by Universal CNPq Grant 408834/2023-4}

\begin{document}

\begin{abstract}
We prove that for any Salem number $\lambda$, totally real number field $k\subseteq\Q(\lambda+\lambda^{-1})$, and positive integer $n\geq\deg_k(\lambda)-1$, there exist infinitely many pairwise incommensurable arithmetic hyperbolic $n$-manifolds defined over $k$ which contain a geodesic of length $\log\lambda$. Additionally, we characterize the rational length spectrum of arithmetic hyperbolic $n$-manifolds of simplest type in terms of the local arithmetic behavior of the associated quadratic form. 
\end{abstract}

\maketitle

\section{Introduction}\label{s:intro}
A \emph{Salem number} is a real algebraic integer $\lambda>1$ with the property that all its Galois conjugates except $\lambda^{-1}$ have norm 1.  A polynomial $f\in\Z[x]$ is a \emph{Salem polynomial} if it is the minimal polynomial of a Salem number. Besides the two real roots $\lambda$ and $\lambda^{-1}$ of a Salem polynomial, the remaining roots occur in reciprocal complex-conjugate pairs $e^{i\theta}$ and $e^{-i\theta}$. Salem polynomials are therefore monic of even degree and self-reciprocal, that is $f(x)=x^{\deg (f)}f(x^{-1})$.
Although Salem numbers are sometimes defined to have at least one pair of non-real conjugates, in this article we allow Salem numbers to be quadratic.

Salem numbers appear in many areas of mathematics. They are closely related to lengths of geodesics in arithmetic hyperbolic manifolds of simplest type. These manifolds are commensurable with quotients of hyperbolic space $\H^n$ by some $\SO(q)_{\Ok}$, the $k$-integer points preserving an admissible quadratic form defined over a totally real number field $k$. If $\ell(\gamma)$ is the translation length of a loxodromic element $\gamma$ in such a $\SO(q)_{\Ok}$, then $\lambda=e^{\ell(\gamma)}$ is a Salem number. In 2019, Emery, Ratcliffe, and Tschantz established a more direct relationship between Salem numbers and lengths of geodesics in hyperbolic manifolds which are arithmetic of simplest type.

\begin{theorem}\cite{EmeryRatcliffeTschantz} \label{ERT conditions}
    Let $\Gamma\subseteq\Isom(\H^n)$ be an arithmetic lattice of simplest type defined over a totally real number field $k$. Let $\gamma$ be a loxodromic element of $\Gamma$, let $\ell(\gamma)$ be the translation length of $\gamma,$ and let $\lambda=e^{\ell(\gamma)}$. If $n$ is even or $\gamma\in \Gamma^{(2)}$ then 
    \begin{enumerate}
        \item $\lambda$ is a Salem number and $\deg_{k}(\lambda)\leq n+1$
        \item $k$ is a subfield of $\Q(\lambda +\lambda^{-1})$, and $[k:\Q(\lambda+\lambda^{-1})]=\deg_k(\gamma)/2$. 
    \end{enumerate}
\end{theorem}

This result can be viewed as giving necessary conditions on members of the length spectrum. 
The (weak) \emph{length spectrum} of a lattice $\Gamma\subseteq\Isom(\H^n)$ is the set
\[ \mathcal{L}(\Gamma)=\{ \ell(\gamma) \mid \gamma \text{ is a loxodromic element in } \Gamma \}. \]
The length spectrum of a hyperbolic manifold encodes deep geometric and arithmetic information of the manifold. 
If $\lambda$ is a Salem number such that $\log\lambda\in\mathcal{L}(\Gamma)$, we say that $\Gamma$ \emph{realizes} $\lambda$.

We instead consider the \emph{rational length spectrum} of a lattice $\Gamma\subseteq\Isom(\H^n)$
\[ \Q\mathcal{L}(\Gamma)=\{s\cdot\ell(\gamma) \mid s\in\Q \text{ and } \gamma \text{ is a loxodromic element in } \Gamma \} \]
and seek to understand when $\log\lambda\in\Q\mathcal{L}(\Gamma)$.
It is easy to see that two commensurable lattices have the same rational length spectrum. On the other hand, two arithmetic lattices $\Gamma_1,\Gamma_2\subseteq\Isom(\H^n)$  with $\Q\mathcal{L}(\Gamma_1)=\Q\mathcal{L}(\Gamma_2)$ are necessarily commensurable whenever $n\not\equiv 1\bmod4$ (\cite{ReidLengths,CHLR,PrasadRapinchuk,Garibaldi}). Prasad and Rapinchuk constructed examples of incommensurable pairs of arithmetic hyperbolic manifolds in each dimension $n\equiv1\bmod4$, $n\geq5$ having the same rational length spectrum \cite{PrasadRapinchuk}.

Emery, Ratcliffe, and Tschantz also gave a converse to \Cref{ERT conditions}.

\begin{theorem}\cite{EmeryRatcliffeTschantz}
    Let $\lambda$ be a Salem number, let $k$ be a subfield of $\Q(\lambda+\lambda^{-1})$, and let $n\geq 2$ be an integer with $\deg_k(\lambda)\leq n+1$. Then there exists an arithmetic lattice of simplest type $\Gamma\subseteq\Isom(\H^n)$ defined over $k$ that realizes $\lambda$.
\end{theorem}

For each Salem number $\lambda$ and number field $k\subseteq\Q(\lambda+\lambda^{-1})$, the arguments in \cite{EmeryRatcliffeTschantz} construct a single arithmetic lattice $\Gamma$ defined over $k$ in each allowable dimension $n$ that realize $\lambda$. The goal of this article is to construct infinitely many pairwise incommensurable arithmetic lattices defined over $k$ in each allowable dimension $n$ that realize $\lambda$.

\begin{main} \label{main theorem}
    Let $\lambda$ be a Salem number, let $k$ be a subfield of $\Q(\lambda+\lambda^{-1})$, and let $n\geq 2$ be an integer with $\deg_k(\lambda)\leq n+1$. Then there exist infinitely many commensurability classes of arithmetic lattices in $\Isom(\H^n)$ defined over $k$ which realize $\lambda$.
\end{main}

We remark that the case of $n=2$ and $k=\Q(\lambda+\lambda^{-1})$ was already known and a proof can be found in \cite[Corollary 12.2.10]{MRBook}. 

An important ingredient at the heart of the realization of $\lambda$ in these arithmetic lattices is a Hasse principle result of Bayer-Fluckiger \cite{Bayer-Fluckinger} that gives criteria for the existence of isometries with prescribed eigenvalues in a quadratic space in terms of the local behavior of both the characteristic polynomial and the quadratic form. We interpret these results in our setting to state a finite set of necessary and sufficient conditions for the realization of $\lambda$ in \Cref{thm conditions}. As a corollary we obtain a description of the elements in $\Q\mathcal{L}(\Gamma)$. To distinguish between commensurability classes of these arithmetic lattices, we use Maclachlan's parametrization of commensurability classes of arithmetic lattices of simplest type in $\Isom(\H^n)$ in terms of the local behavior of the Hasse invariants of the defining quadratic forms \cite{Mac11}. 

The idea of the argument is to use Class Field Theory to construct an infinite set of prime ideals in $k$ that satisfy some explicit splitting behavior in certain field extensions of $k$, and then use distinct finite subsets to construct infinitely many arithmetic lattices defined from quadratic forms over $k$ with arithmetic data determined by the finite set of prime ideals. The prescribed local behavior is specifically chosen with two goals in mind: to guarantee the realization of $\lambda$ by these arithmetic lattices and to guarantee pairwise incommensurability between lattices constructed from distinct subsets of prime ideals. 

The article is organized as follows. In \Cref{s:background} we present the necessary background material on quadratic forms, quaternion algebras, loxodromic isometries, and arithmetic lattices. We also recall Maclachlan's parametrization and state an easy consequence particularly suited for our purposes.
In \Cref{s:realizing} we gather the criteria given by Bayer-Fluckiger which will be necessary in our case. In \Cref{sec:criteria} we give in \Cref{thm conditions} an explicit set of finite conditions on a commensurability class of arithmetic lattices of simplest type in $\Isom(\H^n)$ to realize a Salem number $\lambda$. Finally, we prove \Cref{main theorem} in \Cref{s:infcommclasses}. 

\section{Preliminaries}\label{s:background}

\subsection{Quadratic forms over number fields}
Let $k$ be a field. A \emph{quadratic form} over $k$ is a homogeneous polynomial in $r$ variables of degree 2 with coefficients in $k$. It can be written as 
\[ q(x)=\sum_{i=1}^r \sum_{j=i}^r a_{ij}x_i x_j = x^t S_q x, \] where $S_q$ is the symmetric matrix
    \[ \begin{pmatrix}
    a_{11} & \frac{1}{2}a_{12} & \cdots & \frac{1}{2}a_{1r} \\
    a_{12} & a_{22} & & \vdots \\
    \vdots &  & \ddots & \frac{1}{2}a_{r-1,r} \\
    \frac{1}{2}a_{1r} & \cdots & \frac{1}{2}a_{r-1,r} & a_{rr}
    \end{pmatrix} \]
with diagonal entries $(S_q)_{ii}=a_{ii}$ and off-diagonal entries $(S_q)_{ij}=\frac{1}{2}a_{ij}$. Let $V=k^r$, then the \emph{quadratic space} $(V,q)$ has isometry group given by the orthogonal group 
\[ \Or(q)=\Or(q)_k=\{ A \in \GL(r)_k \:|\: A^tS_qA=S_q \}. \]

The \emph{rank} of $q$ is the number of variables $\rank(q)=r$. The \emph{determinant} of $q$ is $\det(q)=\det(S_q)$. We will assume that $q$ is nondegenerate, that is, $S_q$ is invertible, and $\det(q)\not=0$.
In what follows, we will only care about quadratic forms up to $k^\times$-scaling. 
If $r$ is the rank of $q$, the \emph{discriminant} of $q$ is $\disc(q)=(-1)^{r(r-1)/2}\det(q)$ as an element of $\kk$.

We say that a quadratic form $q$ \emph{represents} $a\in k$ if there is a vector $v\in K^n$ satisfying $q(v)=a$. A quadratic form is \emph{isotropic} if it represents 0 by a nonzero vector and is \emph{anisotropic} otherwise.

Two quadratic forms $q$ and $q'$ over $k$ are \emph{$k$-equivalent} if there exists a matrix $M\in\GL(r)_k$ such that $M^t S_q M= S_{q'}$. When the field $k$ is clear, we write $q\simeq q'$. Two quadratic forms $q$ and $q'$ over $k$ are \emph{$k$-similar} if there is some $a\in k^\times$ such that $q\simeq a q'$.
Any quadratic form over $k$ is $k$-equivalent to a diagonal quadratic form $q'$, which means $q'(x)=\sum_{i=1}^r a_ix_i^2$ and we write $q'=\langle a_1,\dots,a_r \rangle$.

Over the reals, any nondegenerate quadratic form is $\R$-equivalent to a diagonal form with coefficients in $\pm1$. The \emph{signature} of $q$ is the pair $(n_+, n_-)$ where $n_+$ is the number of $+1$'s and $n_-$ is the number of $-1$'s. 

Let $k$ be a number field and $\Ok$ its ring of integers.
We denote by $\Omega$ the set of all places $\nu$ of $k$. The subset of infinite places and finite places will be denoted by $\Omega_\infty$ and $\Omega_{fin}$ respectively. Since the set of finite places of $k$ is in 1-to-1 correspondence with the set of prime ideals in $\Ok$, we will abuse notation and often write $\mathfrak{p}$ to refer either to prime ideals or to their corresponding finite places. 

If $\nu$ is a place of $k$, let $k_\nu$ denote the completion corresponding to $\nu$ and let $q_\nu$ denote the quadratic form over $k_\nu$ obtained from $q$ by $\nu$. 
If $q$ is defined over a real number field $k\hookrightarrow \R$, we refer to the signature with respect to the given embedding as the \emph{signature} of $q$ over $k$.

\subsection{Quaternion algebras and invariants of quadratic forms}

A \emph{quaternion algebra} over a field $k$ is a four-dimensional central simple algebra. Any quaternion algebra $B$ can be represented by a Hilbert symbol $(\frac{a,b}{k})$ for some $a,b\in k^\times$ where $B$ has a basis $\{1,i,j,ij\}$ with $a,b\in k^\times$, $i^2=a$, $j^2=b$, and $ij=-ji$.

Any quaternion algebra is either isomorphic to the matrix algebra $M_2(k)$ or to a division algebra. Over $\R$ or over any local field, there is a unique isomorphism class of division quaternion algebras.
Furthermore, $B=(\frac{a,b}{k})\cong M_2(k)$ if and only if the quadratic form $ax^2+by^2-z^2$ is isotropic.

Given $B$ a quaternion algebra over a number field $k$ and $\nu$ a place of $k$, let $B_\nu=B\otimes_k k_\nu$, a quaternion algebra over $k_\nu$. 
A quaternion algebra $B$ over $k$ is said to be \emph{split} at a place $\nu$ if $B_\nu\cong M_2(k_\nu)$ and is otherwise said to be \emph{ramified} at $\nu$ if $B_\nu$ is a division algebra. 

The ramification set of the quaternion algebra $B$ is denoted
    \[ \Ram(B) =\{\nu\in\Omega \: | \: B \text{ is ramified at }\nu  \}. \]
The subset of infinite ramified places will be denoted by $\Ram_\infty(B)$. The subset of finite ramified places will be denoted $\Ram_{fin}(B)$. A Hilbert symbol of a quaternion algebra is not unique. However, quaternion algebras are classified by their ramification sets according to the following theorem (see e.g. \cite[Theorem 7.3.6]{MRBook}).

\begin{theorem}[Classification of quaternion algebras]
    Let $B$ be a quaternion algebra over a number field $k$.
    The following hold:
    \begin{itemize}
        \item $\Ram(B)$ is a finite set of even cardinality.
        \item Two quaternion algebras $B_1$ and $B_2$ defined over $k$ are isomorphic if and only if $\Ram(B_1)=\Ram(B_2)$.
        \item If $S$ is any finite subset of $\Omega$ of even cardinality, excluding any complex places, then there exists a quaternion algebra $B$ over $k$ such that $\Ram(B)=S$.
    \end{itemize}
\end{theorem}

Isomorphism classes of quaternion algebras over $k$ form a subgroup of the Brauer group $\Br(k)$, where every element has order 2. The isomorphism class of the quaternion algebra with empty ramification set corresponds to the identity element in $\Br(k)$. In $\Br(k)$, multiplication of quaternion algebras is given by $B_1\cdot B_2=B_3$ where
\[
    \Ram(B_3) = \left(  \Ram(B_1) \cup  \Ram(B_2) \right) - \left(  \Ram(B_1) \cap  \Ram(B_2) \right) .
\]

Let $(V,q)$ be a quadratic space over a number field $k$ with diagonalization $\langle a_1,\dots a_r\rangle$. The \emph{Hasse invariant} of $(V,q)$ is the product
\[ s(q)=\prod_{i<j} \left( \tfrac{a_i,a_j}{k} \right) \in \Br(k), \]
which is independent of the diagonalization, and is invariant in the $k$-equivalence class. This invariant is closely related to the \emph{Witt invariant} $c(q)$, which can be obtained as follows:
\begin{equation}\label{eq:WittHasse}
\begin{aligned}
    \text{if }\rank(q)\equiv 1,2\bmod8, \: & \: c(q) = s(q)  \\
    \text{if }\rank(q)\equiv 3,4\bmod8, \: & \: c(q) = s(q)\cdot \left( \tfrac{-1, -\det(q)}{k} \right) \\
    \text{if }\rank(q)\equiv 5,6\bmod8, \: & \: c(q) = s(q)\cdot \left( \tfrac{-1,-1}{k} \right) \\
    \text{if }\rank(q)\equiv 7,8\bmod8, \: & \: c(q) = s(q)\cdot \left( \tfrac{-1, \det(q)}{k} \right)  .
\end{aligned}
\end{equation}

\subsection{Hyperbolic space and loxodromic isometries} 
\label{ss:hyperbolic isometries}

The hyperbolic $n$-space $\H^n$ is the simply connected Riemannian $n$-manifold in which the sectional curvature is constant $-1$. We consider the hyperboloid model of hyperbolic space defined as follows.

Let $n \geq 2$ and $q=\langle 1,\dots , 1,-1\rangle$ be the standard quadratic form of signature $(n,1)$. The bilinear form $B$ in $\R^{n+1}$ associated with $q$ is given by
\[ B(x,y) = \frac{q(x+y)-q(x)-q(y)}{2}. \]
The set \[ \mathcal{S}=\{(x_1,\ldots,x_{n+1}) \in \R^{n+1}; q(x_1,\ldots,x_{n+1})=-1 \} \] forms a two-sheeted hyperboloid. We denote the positive sheet by \[\mathcal{S}^+=\{(x_1,\ldots,x_{n+1}) \in \mathcal{S}; x_{n+1}>0 \}.\] 
The \emph{hyperboloid model} of the hyperbolic $n$-space is obtained by identifying $\H^n$ with the $n$-manifold $\mathcal{S}^+$ together with the Riemannian metric given by $d(x,y)=\arccos\left(B(x,y)\right)$. 

The orthogonal group $\Or(n,1)_\R$ preserves the hyperboloid $\mathcal{S}$. 
We identify the isometry group of $\H^n$ with 
\[ \Isom(\H^n)=\mathrm{PO}(n,1)_\R\cong \SO(n,1)_\R . \]
It has two connected components. The identity component consists of orientation-preserving isometries. The other component consists of orientation-reversing isometries.

We note that one may construct isometric hyperboloid models for $\H^n$ using any quadratic form $q'$ with signature $(n,1)$ defined over a real field. Since any such $q'$ is $\R$-equivalent to the standard form $q=\langle 1,\dots , 1,-1\rangle$, there is some $g\in\GL(n+1)_\R$ such that $g^t S_q g=S_{q'}$ and $g^{-1}\SO(n,1)_\R g= \SO(q)_\R$.

An element $\gamma\in \Isom(\H^n)$ is called \emph{loxodromic} if there is a unique geodesic $L\in\H^n$, called the \emph{axis} of $\gamma$, on which $\gamma$ acts as a translation by a positive distance $\ell(\gamma)$, called the \emph{translation length} of $\gamma$. In this case, $\gamma$ has exactly $n-1$ eigenvalues of norm one, and two real eigenvalues $\lambda=e^{\ell(\gamma)}>1$ and $\lambda^{-1}=e^{-\ell(\gamma)}<1$ not on the unit circle (see \cite{ChenGreenberg}).

A \emph{lattice} is a discrete subgroup $\Gamma<\Isom(\H^n)$ such that the quotient $\H^n/\Gamma$ has finite volume. Whenever $\Gamma$ is a lattice, the quotient $\H^n/\Gamma$ is a complete finite-volume hyperbolic orbifold. In this case, most elements of $\Gamma$ are loxodromic and the axes of loxodromic elements map to closed geodesics in the quotient. Also, $\H^n/\Gamma$ will be a manifold whenever $\Gamma$ is torsion free. 

Two lattices $\Gamma_1$ and $\Gamma_2$ in $\Isom(\H^n)$ are \emph{commensurable} if there exists $g\in\Isom(\H^n)$ such that $\Gamma_1\cap g\Gamma_2g^{-1}$ has finite index in both $\Gamma_1$ and $g\Gamma_2g^{-1}$.

\subsection{Arithmetic lattices}

Let $k$ be a totally real number field together with a chosen real embedding $\nu_0:k\hookrightarrow\R$. Consider a quadratic form $q:\R^{n+1} \rightarrow \R$ defined over $k$, with signature $(n,1)$ with respect to the $\nu_0$.

We say that such a form $q$ is \emph{admissible} if for any non-identity real place $\nu$ of $k$, $q_\nu$ has signature $(n,0)$ (in other words, $q_\nu$ is positive definite). This condition on the Galois conjugates of $q$ guarantees that the subgroup
\[ \SO(q)_{\Ok} = \{ A \in \SL(n+1)_{\Ok}; A^tS_qA=S_q\} \]
is a discrete subgroup of $\SO(q)_\R \cong \Isom(\H^n)$. Furthermore, $\SO(q)_{\Ok}$ is a lattice in $\Isom(\H^n)$.

Any lattice $\Gamma<\Isom(\H^n)$ commensurable with some $\SO(q)_{\Ok}$ for an admissible $q$ defined over a totally real number field $k$ is called \emph{arithmetic of simplest type}. The quotient orbifold $\H^n/\Gamma$ is also called arithmetic of simplest type.

When $n$ is even, every arithmetic lattice in $\Isom(\H^n)$ is of simplest type and contained in $\SO(q)_k$ for some admissible $q$ defined over a totally real number field $k$ (see e.g. \cite[Lemma 4.2]{EmeryRatcliffeTschantz}). When $n$ is odd, there exist arithmetic lattices that are not of simplest type; however, if $\Gamma$ is arithmetic of simplest type, then $\Gamma^{(2)}$, the subgroup generated by the squares of the elements in $\Gamma$, is contained in $\SO(q)_k$ for some admissible $q$ defined over a totally real number field $k$ (see \cite[Lemma 4.5]{EmeryRatcliffeTschantz}).

The commensurability classes of these discrete subgroups of $\Isom(\H^n)$ are in one-to-one correspondence with similarity classes of admissible quadratic forms $q$ over totally real number fields $k$.
In \cite{Mac11}, Maclachlan classified the commensurability classes of arithmetic lattices of simplest type in terms of the Witt invariant of quadratic forms. The following statement is a direct consequence of the proofs of \cite[Theorems 7.2 and 7.4]{Mac11}.

\begin{proposition}\cite{Mac11} \label{Maclachlan}
    Let $\Gamma=\SO(q)_{\Ok}$ and $\Gamma'=\SO(q')_{\Ok}$ be arithmetic lattices of the simplest type in $\Isom(\H^n)$, where $q$ and $q'$ are two admissible quadratic forms defined over a totally real number real field $k$.

    If $n$ is even, then $\Gamma$ and $\Gamma'$ are commensurable if and only if $c(q)=c(q')$.

    If $n$ is odd, let $\delta=\disc(q)$ and $\delta'=\disc(q')$.
    Then $\Gamma$ and $\Gamma'$ are commensurable if and only if $\delta=\delta'$ and $c(q)\otimes_k k(\sqrt{\delta})=c(q')\otimes_k k(\sqrt{\delta})$ in $\Br\left(k(\sqrt{\delta})\right)$.
\end{proposition}

Let $\{\mathfrak{p}_1,\dots,\mathfrak{p}_s\}$ denote the subset of prime ideals in $k$ corresponding to the finite places in $\Ram\left(c(q)\right)$ which split in $k(\sqrt{\delta})$.
Then the finite places at which $c(q)\otimes_k k(\sqrt{\delta})$ is ramified correspond to the $2s$ prime ideals in $k(\sqrt{\delta})$ lying over $\{\mathfrak{p}_1,\dots,\mathfrak{p}_s\}$ \cite[\S7]{Mac11}. We get the following condition which we will use to distinguish commensurability classes. 

\begin{lemma} \label{lem:Mac odd non comm}
    Suppose that $\SO(q)_{\Ok}$ and $\SO(q')_{\Ok}$ are commensurable for admissible quadratic forms $q$ and $q'$ of signature $(n,1)$ defined over $k$ with $n$ odd. If $\nu_\mathfrak{p}\in\Ram\left(c(q)\right)$ corresponds to a prime ideal $\mathfrak{p}$ in $k$ which splits in $k(\sqrt{\delta})$, then $\nu_\mathfrak{p}\in\Ram\left(c(q')\right)$.
\end{lemma}

\section{Realizing isometries in quadratic spaces} \label{s:realizing}

Bayer-Fluckiger proved the following Hasse principle for the existence of an isometry of a given characteristic polynomial

\begin{theorem}\cite{Bayer-Fluckinger}\label{thm BF main}
     Suppose that $k$ is a global field. A quadratic space $(V,q)$ over $k$ has an isometry with characteristic polynomial $F$ if and only if such an isometry exists over all the completions of $k$.
\end{theorem}

In this section we gather more precise conditions given in \cite{Bayer-Fluckinger} for a quadratic space to admit an isometry with a given characteristic polynomial.

\begin{definition}
    A monic polynomial $F\in k[x]$is said to be $\epsilon$-symmetric for some $\epsilon=\pm 1$ if $F(x)=\epsilon x^{\deg(F)} F(x^{-1})$. If $\epsilon=1$, we say that $F$ is symmetric.
\end{definition}

Characteristic polynomials of isometries of quadratic spaces are $\epsilon$-symmetric, with $\epsilon$ equal to the constant term of the polynomial (see \cite[Proposition 11]{Bayer-Fluckinger}). If $F\in k[x]$ is a monic polynomial such that $F(0)\neq 0$, set 
\[ F^{*}(x)=\frac{1}{F(0)}x^{\deg(F)}F(x^{-1}). \]

Note that $F^*$ is also monic, $F^*(0)\neq0,$ and $F^{**}=F.$ 

\begin{definition}
    Let $F\in K[x]$ be a monic, $\epsilon$-symmetric polynomial. We say that $F$ is of
    \begin{enumerate}[left=10pt]
        \item[type 0:] if $F$ is a product of powers of $x-1$ and $x+1$.
        \item[type 1:] if $F$ is a product of powers of monic, symmetric, irreducible polynomials in $ k[x]$ of even degree.
        \item[type 2:] if $F$ is a product of polynomials of the form $gg^*$, where $g\in k[x]$ is monic, irreducible, and $g\neq g^*$.
    \end{enumerate}
\end{definition}

\begin{proposition}\cite[Prop 1.3]{Bayer-Fluckinger} \label{BF symm poly types}
Every monic, $\epsilon$-symmetric polynomial $F\in k[x]$ is a product of polynomials of type 0, 1, and 2.
\end{proposition}

A monic $\epsilon$-symmetric polynomial is \textit{hyperbolic} if all its components of type 0 and 1 are of the form $f^e$ with $e$ even.

Let $q$ be a quadratic form over $k$ and $F\in k[x]$ be a symmetric polynomial with $\deg(F)=\rank(q)$. For every real place $\nu$ of $k$, let $(r_\nu,s_\nu)$ denote the signature of $q_\nu$, and let $\sigma_\nu$ be the number of pairs of roots of $F\in k_\nu[x]$ that are not on the unit circle.   
Then $F$ and $q$ are said to satisfy the \emph{signature condition} if for every real place $\nu$ of $k$
$r_\nu\geq\sigma_\nu$, $s_\nu\geq\sigma_\nu$, and  whenever $F$ has no type 0 factors $(r_\nu,s_\nu)\equiv(\sigma_\nu,\sigma_\nu)\bmod 2$.

If the rank of $q$ is even and $F$ is as above with $\deg(F)=\rank(q)=2m$, let $S$ be the set of places of $k$ at which the Hasse invariant of $q$ is different from the Hasse invariant of the $2m$-dimensional space associated with the quadratic form $\langle1,\dots,1,-1,\dots,-1\rangle$ of signature $(m,m)$. By the classification of quaternion algebras, this set $S$ is finite.
Then $F$ and $q$ are said to satisfy the \emph{hyperbolicity condition} if whenever $\nu\in S$, then $F\in k_\nu[x]$ is not hyperbolic. (Note that this definition varies slightly from the definition in \cite{Bayer-Fluckinger} but is implied by \cite[Corollary 9.3]{Bayer-Fluckinger}).

For later use, we note here that the Hasse invariant of the $2m$-dimensional space associated, with the quadratic form $\rho^m_m=\langle1,\dots,1,-1,\dots,-1\rangle$ of signature $(m,m)$, is given by
\begin{equation} \label{eq:2m hyp form}
    s(\rho^m_m)= \begin{cases}
    M_2(k) &\text{if } m\equiv 0,1\bmod 4 \\
    \left( \tfrac{-1,-1}{k} \right) &\text{if } m\equiv 2,3\bmod 4 . \end{cases}
\end{equation}

\begin{theorem}\cite[Corollary 9.3]{Bayer-Fluckinger}\label{BF f0 0} Suppose that $k$ is a global field, $F\in k[x]$ is an irreducible polynomial of type 1, and $q$ is a quadratic form with $\rank(q)=\deg(F)$. The quadratic space $(V,q)$ has an isometry with characteristic polynomial $F$ if and only if the following conditions are satisfied:
\begin{enumerate}[left=10pt]
    \item[(i)] $\det(q)\equiv F(1)F(-1)$ in $\kk$, 
    \item[(ii)] the signature condition holds,
    \item[(iii)] the hyperbolicity condition holds.
\end{enumerate}
\end{theorem}

Bayer-Fluckiger also established in \cite{Bayer-Fluckinger} conditions for a polynomial of mixed type to be the characteristic polynomial in a quadratic space $(V,q)$. We state the necessary results therein, applied to our cases of interest. We will consider polynomials of the form $F(x)=f(x)\cdot f_0(x)$ with $f$ an irreducible polynomial of type 1 and $f_0$ a polynomial of type 0.  

\begin{proposition}\cite[Proposition 11.11]{Bayer-Fluckinger} \label{Bf f0 deg 1}
    Suppose that $k$ is a global field, $F(x)=f(x)\cdot f_0(x)\in k[x]$ is the product of an irreducible polynomial $f$ of type 1 and a degree 1 polynomial $f_0$ of type 0. Let $q$ be a quadratic form of odd dimension with $\rank(q)=\deg(F)=\deg(f)+1$.
    The quadratic space $(V,q)$ has an isometry with characteristic polynomial $F$ if and only if $q\simeq \langle d_0\rangle\oplus q'$ where $d_0=\det(q)f(1)f(-1)$ and $(V',q')$ has an isometry with characteristic polynomial $f$.
\end{proposition}

The \emph{Witt index} of a quadratic form $q$ is the largest integer $m$ such that $q$ has a decomposition with $\rho^m_m$ (as in \Cref{eq:2m hyp form}) as a summand.

\begin{proposition}\cite[Proposition 11.9]{Bayer-Fluckinger} \label{Bf f0 deg 2}
    Suppose that $k$ is a global field, $F(x)   =f(x)\cdot f_0(x)\in k[x]$ is the product of an irreducible polynomial $f$ of type 1 and a degree 2 polynomial $f_0$ of type 0. Let $q$ be a quadratic form of even rank with $\rank(q)=\deg(F)=\deg(f)+2$.
    The quadratic space $(V,q)$ has an isometry with characteristic polynomial $F$ if and only if the following conditions are satisfied:
\begin{enumerate}[left=10pt]
    \item[(i)] the signature condition holds
    \item[(ii)] if $\nu$ is a finite place of $k$ where $f\in k_\nu[x]$ is hyperbolic, then the Witt index of $q$ defined over  $k_\nu$ is $\geq \frac{1}{2}\deg(f)$.
\end{enumerate}
\end{proposition}

\begin{proposition}\cite[Proposition 11.2]{Bayer-Fluckinger} \label{BF f0 deg3}
    Suppose that $k$ is a global field, $F(x)=f(x)\cdot f_0(x)\in k[x]$ is the product of an irreducible polynomial $f$ of type 1 and a polynomial $f_0$ of type 0 having $\deg(f_0)\geq 3$. Let $q$ be a quadratic form with $\rank(q)=\deg(F)$.
    The quadratic space $(V,q)$ has an isometry with characteristic polynomial $F$ if and only if the signature condition is satisfied.
\end{proposition}

\section{Necessary and sufficient criteria}\label{sec:criteria}

In this section we find necessary and sufficient conditions for the realization of $\lambda$ in the commensurability class of  $\SO(q)_{\Ok}$ in terms of the rank, discriminant, and Hasse invariant of an admissible quadratic form $q$ defined over $k$. Note that for an admissible quadratic form $q$, the Hasse invariant splits at all real places.

We fix the following notation for the remainder.
\begin{itemize}
    \item $n\geq2$.
    \item $\lambda$ a Salem number with minimal polynomial $f_\lambda\in\Z[x]$ over $\Q$.
    \item $k\subseteq \Q(\lambda+\lambda^{-1})$ a totally real number field with embedding $\nu_0:k\hookrightarrow\R$.
    \item $f(x)\in k[x]$ the minimal polynomial of $\lambda$ over $k$.
    \item for $s\in\N$, $f_s(x)\in k[x]$ the minimal polynomial of $\lambda^s$ over $k$.
    \item The number field $K=\Q(\lambda)=k(\lambda)\cong k[x]/(f)$ with splitting field $K'$.
    \item $E=\Q(\lambda+\lambda^{-1})=k(\lambda+\lambda^{-1})$ with splitting field $E'$.
    \item $D=f(1)f(-1)$.
    \item $\Sigma^{ns}(K)$ the set of places $\nu$ of $k$ for which there exists a place $\mathfrak{v}$ of $E$ lying over $\nu$ which does not split in $K/E$.
\end{itemize}

The following theorem gives an explicit set of conditions on an admissible quadratic form $q$ defined over a totally real number field $k$ to determine whether or not a Salem number is realized in the commensurability class of the arithmetic lattice $\SO(q)_{\Ok}$. 

\begin{main} \label{thm conditions}
    Let $\lambda$ be a Salem number and let $q$ be an admissible quadratic form defined over the totally real number field $k$. Then $\log\lambda\in\Q\mathcal{L}(\SO(q)_{\Ok})$ exactly when $k\subseteq\Q(\lambda+\lambda^{-1})$ and one of the following is true
    \begin{enumerate}[(i)]
        \item $n+1-\deg_k(\lambda)\geq 3$,
        \item $n+1-\deg_k(\lambda)=2$, $\frac{n-1}{2}\equiv 1,2\bmod4$ and either
            \begin{itemize}
                \item[--] for any $\nu\in\Ram\left(s(q)\right)$ such that $\det(q)\equiv(-1)^{m+1}$ in $k_\nu^\times/k_\nu^{\times 2}$, $\nu\in\Sigma^{ns}(K)$,
                \item[--] or, if there exists a root of unity $\zeta$ such that $(\zeta+\zeta^{-1})\in k$, $\det(q)(4-\left(\zeta+\zeta^{-1})^2\right)\cong D$ in $\kk$ and for any $\nu\in\Ram\left(s(q)\cdot \left(\tfrac{-1,-1}{k}\right) \right)$ if $\nu$ splits in $k(\zeta)/k$ then $\nu\in\Sigma^{ns}(K)$,
            \end{itemize}
        \item $n+1-\deg_k(\lambda)=2$, $\frac{n-1}{2}\equiv 0,3\bmod4$, and either
            \begin{itemize}
                \item[--] for any $\nu\in\Ram\left(s(q)\cdot \left(\tfrac{-1,-1}{k}\right) \right)$ such that $\det(q)\equiv(-1)^{m+1}$ in $k_\nu^\times/k_\nu^{\times 2}$, $\nu\in\Sigma^{ns}(K)$,
                \item[--] or, if there exists a root of unity $\zeta$ such that $(\zeta+\zeta^{-1})\in k$, $\det(q)(4-\left(\zeta+\zeta^{-1})^2\right)\cong D$ in $\kk$ and for any $\nu\in\Ram\left(s(q) \right)$ if $\nu$ splits in $k(\zeta)/k$ then $\nu\in\Sigma^{ns}(K)$,
            \end{itemize}
        \item $n+1-\deg_k(\lambda)=1$, $ \frac{n}{2}\equiv 0,1\bmod4$, and $\Ram\left(s(q)\right)\subset\Sigma^{ns}(K)$.
        \item $n+1-\deg_k(\lambda)=1$, $ \frac{n}{2}\equiv 2,3\bmod4$, and $\Ram\left(s(q)\cdot \left(\tfrac{-1,-1}{k}\right)\right) \subseteq\Sigma^{ns}(K)$,
        \item $n+1-\deg_k(\lambda)=0$, $\frac{n+1}{2}\equiv 0,1\bmod4$, $\det(q)\equiv f(1)f(-1)$ in $\kk$, and $\Ram\left(s(q)\right)\subset\Sigma^{ns}(K)$,
        \item or $n+1-\deg_k(\lambda)=0$, $\frac{n+1}{2}\equiv 2,3\bmod4$, $\det(q)\equiv f(1)f(-1)$ in $\kk$, and $\Ram\left(s(q)\cdot \left(\tfrac{-1,-1}{k}\right)\right)\subseteq\Sigma^{ns}(K)$.
    \end{enumerate}
\end{main}

The proof of \Cref{thm conditions} is divided according to the difference $n+1-\deg_k(\lambda)$ in \Cref{prop:dif3 conditions,prop:dif 2 conditions,prop: dif 1 conditions,prop:dif 0 conditions}.
We obtain as an immediate consequence the following description of 
the rational length spectrum for an arithmetic lattice in $\Isom(\H^n)$ of simplest type.

\begin{corollary}
    \[ \Q\mathcal{L}(\SO(q)_{\Ok})= \{ r\log\lambda \mid r\in\Q \text { and $\lambda$ satisfies the conditions in \Cref{thm conditions}} \}. \]
\end{corollary}

In order to realize $\lambda$, we will need to find an isometry in $\SO(q)_{\Ok}$ whose characteristic polynomial contains $f$ as a factor. Any other factor will need to also be monic, symmetric, and have all its roots and Galois conjugates of roots on the unit circle, hence roots of unity.
The applications of results by Bayer-Fluckiger in the previous section give us conditions for the existence of isometries in $\Or(q)_{\Ok}$ realizing a desired characteristic polynomial. The following lemma realizes such isometries in the integer points, up to $k$-equivalence.

\begin{lemma}\label{lem:ktoRk}
    Let $T$ be an isometry in $\Or(q)_k$ where $q$ is an admissible quadratic form over $k$. Suppose also that its characteristic polynomial $p_T(x)$ has coefficients in the ring of integer $\Ok$. Then a conjugate of $T$ is contained in some $\Or(q')_{\Ok}$ where $q'\simeq q$.
\end{lemma}

\begin{proof}
    Let $r$ denote the rank of $q$, which is also the degree of $p_T\in \Ok[x]$. Denote the coefficients of $p_T$ by $c_0=1,c_1,\dots c_{r-1},c_r=1$. The matrix powers $T,T^2,\dots,T^{r-1}$ all have entries in $k$. 
    Choose $\mathfrak{m}$ to be a common multiple of the primes appearing in the denominators of the entries in these $r-1$ matrices.
    By the Cayley-Hamilton theorem $T$ satisfies its characteristic polynomial, so for any $a\in\Z_{\geq0}$
    \[ T^{a+r}=-\left( T^a +c_1 T^{a+1}+\cdots +c_{r-1}T^{a+r-1} \right) . \]
    
    The subgroup of $k^r$ generated by the images $T^a (\Ok^r)$ for all $a\geq 0$ lies between the subgroups $\Ok^r$, and $(\frac{1}{\mathfrak{m}}\Ok)^r$ and is thus itself an $\Ok$-lattice of rank $r$. We can therefore choose a change of basis matrix $g\in\GL(r)_k$ such that $g^{-1}Tg$ has entries in $\Ok$.
    
    If $S_q$ is the symmetric matrix associated to the quadratic form $q$, let $q'$ be the quadratic form associated to the symmetric matrix $S_{q'}=g^t S_q g$. Then $g^{-1} T g\in \Or(q')_{\Ok}$ and since $q\simeq q'$, $\Or(q)_{\Ok}$ and $\Or(q')_{\Ok}$ are commensurable.
\end{proof}

We will need the following lemma regarding $D=f(1)f(-1)$.

\begin{lemma} \label{lem:determinant neg}
     $D<0$ and if $\sigma\neq\nu_0$ is a real embedding $\sigma:k\hookrightarrow\R$, then $\sigma (D)>0$.
\end{lemma}

\begin{proof}
    By \cite[Lemmas 6.1 and 6.2]{EmeryRatcliffeTschantz}, $2m=\deg(\lambda)=\deg_k(\lambda)\cdot [k:\Q]$, $\deg_k(\lambda)=2m_k$ for some integer $m_k|m$.
    We have that $f|f_\lambda$ and the roots of $f$ are a subset of the reciprocal pairs of roots of $f_\lambda$. 
    Over $\R$, $f_\lambda$ factors as in \Cref{eq:R factor} and then $f$ factors a subset of $m_k$ of these, including $f_0=\left(x^2-(\lambda+\lambda^{-1})x+1\right)$. All other factors of $f_\lambda$  have the form $f_i=(x^2-2\cos\theta_ix+1)$ for $i=1,\dots,m-1$. Clearly $f_0(1)f_0(-1)<0$ and for $i\not=0$ $f_i(1)f_i(-1)>0$. A real embedding $\sigma\neq\nu_0$ will send all the factors of $f$ to factors in $f_\lambda/f$. So $f(1)f(-1)<0$ and $\sigma \left( f(1)f(-1) \right)>0$.
\end{proof}

We now consider each case of $n+1-\deg_k(\lambda)=n+1-\deg(f)$. The easiest case is when this difference is at least 3, while the most difficult case is when this difference is 2.

\begin{proposition} \label{prop:dif3 conditions}
    If $n+1-\deg(f)\geq 3$ and $q$ is an admissible quadratic form over $k$ with signature $(n,1)$ then $\log\lambda\in\Q\mathcal{L}(\SO(q)_{\Ok})$.
\end{proposition}

\begin{proof}
    Consider an arbitrary admissible quadratic from $q$ over $k$ with signature $(n,1)$. Let $F(x)=f(x)\cdot (x+1)^{n+1-\deg(f)}$. By \Cref{BF f0 deg3}, there is an isometry $T\in \Or(q)_k$ with characteristic polynomial $p_T=F$ if and only if the signature condition holds. That the signature condition holds follows from the proof of \Cref{lem:determinant neg}.

    By \Cref{lem:ktoRk}, a conjugate $\gamma$ of $T$ is contained in $\Or(q')_{\Ok}$ for some $q'\simeq q$. Note that $\det(\gamma)=\det(T)=F(0)=1$, so $\gamma\in\SO(q')_{\Ok}$. Thus $\log\lambda\in\mathcal{L}(\SO(q')_{\Ok})\subseteq\Q\mathcal{L}(\SO(q)_{\Ok})$.
\end{proof}

In the remaining cases, we will need to check the hyperbolicity condition. Bayer-Fluckinger proved a equivalent condition for a polynomial to become hyperbolic which will be easier to check in our particular case of $f$, $K$, and $E$.

\begin{lemma}\cite[Lemma 9.4]{Bayer-Fluckinger} \label{lem: BF splitting hyperbolic}
    Let $\nu$ be a place of $k$. Then $f\in k_\nu[x]$ is hyperbolic if and only if $\nu\not\in\Sigma^{ns}(K)$. 
\end{lemma}

Note that $\log\lambda$ is in $\Q\mathcal{L}(\SO(q)_{\Ok})$ if and only if $\SO(q)_{\Ok}$ realizes some power $\lambda^s$. We will need the following lemma.

\begin{lemma}\label{lem:Salem power}
    For any $s\in\N$, let $f_s(x)$ denote the minimal polynomial of $\lambda^s$ over $k$. Then
    \begin{enumerate}
        \item $\lambda^s$ is a Salem number,
        \item $k(\lambda^s)=k(\lambda)=K$,
        \item $f_s(1)f_s(-1)\equiv f(1)f(-1)\equiv D$ in $\kk$,
        \item and for any place $\nu$ of $k$, $f_s\in k_\nu[x]$ is hyperbolic if and only if $f\in k_\nu[x]$ is hyperbolic.
    \end{enumerate}
\end{lemma}

\begin{proof}
    If $\alpha$ is any conjugate of $\lambda$, then clearly $\alpha^s$ is a conjugate of $\lambda^s$. If $\sigma$ is the conjugation taking $\alpha\mapsto\lambda$, then $|\sigma(\alpha^s)|=|\lambda^s|>1$ whereas for any other conjugate $\beta$, $|\sigma(\beta^s)|=|\sigma(\beta)|^s\leq 1$. Thus $\lambda^s$ is Salem of the same degree as $\lambda$ and hence $K=k(\lambda)=k(\lambda^s)$.

    Let $2m=\deg(f)=\deg_k(\lambda)$. As is noted in \cite[(2-4)]{MR2575371} and \cite[p.85]{MR3562700}, the discriminant of $\Delta_f$ of $f$ and that of its trace polynomial $g$ are related by the formula
    \[ \Delta_f=\Delta_g^2(\lambda-\lambda^{-1})^2\prod_{i=1}^{m-1}(e^{i\theta_i}-e^{-i\theta_i})^2. \]
    Note that 
    \[ f(x)= (x^2-(\lambda+\lambda^{-1})x+1)\prod_{i=1}^{m-1}(x^2-(e^{i\theta_i}+e^{-i\theta_i})x+1) \]
    and thus
    \begin{align*}
        f(1)f(-1) &= (2-\lambda-\lambda^{-1})(2+\lambda+\lambda^{-1})\prod_{i=1}^{m-1}(2-e^{i\theta_i}-e^{-i\theta_i})(2+e^{i\theta_i}+e^{-i\theta_i})
        \\ &= (2-\lambda^2-\lambda^{-2}) \prod_{i=1}^{m-1}(2-e^{2i\theta_i}-e^{-2i\theta_i})
        \\ &= (-1)^m(\lambda-\lambda^{-1})^2 \prod_{i=1}^{m-1}(e^{i\theta_i}-e^{-i\theta_i})^2.
    \end{align*} 
    Therefore, 
    \[ \Delta_f=(-1)^m\Delta_g^2 f(1)f(-1)\]
    and hence $\Delta_f\equiv (-1)^m f(1)f(-1)$ in $\kk$. Similarly, $\Delta_{f_s}\equiv (-1)^m f_s(1)f_s(-1)$ in $\kk$. But since $k(\lambda)=k(\lambda^s)$, through the relationship between $\Delta_f$ and the relative field discriminant of $k(\lambda)/k$,
    one sees that $\Delta_f\equiv\Delta_{f_s}$ in $\kk$. 
    It follows that $f_s(1)f_s(-1)\equiv f(1)f(-1)$ in $\kk$.

    Finally, since $K=k[x]/(f)=k[x]/(f_s)$, \cite[Lemma 9.4]{Bayer-Fluckinger} applies equally to $f_s$ as it does to $f$.
\end{proof}

The previous lemma asserts that $f$ and $f_s$ have the same local behavior, as well as the same global behavior up to squares.  We therefore obtain the following consequence combining \Cref{thm BF main} with \Cref{lem:Salem power,lem:ktoRk}.

\begin{corollary}\label{cor:f vs fs}
    $\log\lambda\in\Q\mathcal{L}(\SO(q)_{\Ok})$ if any only if there is some $k$-similar $q'\simeq q$ for which $\log\lambda\in\mathcal{L}(\SO(q')_{\Ok})$.
\end{corollary}

We now consider the case when the difference is 0. In this case, $f$ must be the characteristic polynomial of the isometry realizing $\lambda$.

\begin{proposition} \label{prop:dif 0 conditions}
    Suppose $n+1=\deg(f)$ and $q$ is an admissible quadratic form over $k$ with signature $(n,1)$.
    
    If $ \frac{n+1}{2}\equiv 0,1\bmod4$ then $\log\lambda\in\Q\mathcal{L}(\SO(q)_{\Ok})$ if and only if $\det(q)\equiv f(1)f(-1)$ in $\kk$ and $\Ram\left(s(q)\right)\subset\Sigma^{ns}(K)$.
    
    If $ \frac{n+1}{2}\equiv 2,3\bmod4$ then $\log\lambda\in\Q\mathcal{L}(\SO(q)_{\Ok})$ if and only if $\det(q)\equiv f(1)f(-1)$ in $\kk$ and $\Ram\left(s(q)\cdot \left(\tfrac{-1,-1}{k}\right)\right)\subseteq\Sigma^{ns}(K)$.
\end{proposition}

\begin{proof}
    Let $m=\frac{n+1}{2}$. Up to $k$-similarity, it suffices by \Cref{cor:f vs fs} to consider the realization of $\lambda$ in $\SO(q)_{\Ok}$.

    By \Cref{BF f0 0}, there is an isometry $T$ in $\Or(q)_{k}$ with characteristic polynomial $f(x)$ if and only if the conditions {\it (i)-(iii)} are satisfied. Condition {\it (ii)} is satisfied by the proof of \Cref{lem:determinant neg}.

    Suppose first that $m\equiv 0,1\bmod4$. In this case the Hasse invariant of the quadratic form $\rho^m_m$ has empty ramification (see \Cref{eq:2m hyp form}). Then by \Cref{lem: BF splitting hyperbolic}, condition {\it (iii)} is satisfied if and only if $\Ram\left(s(q)\right)\subseteq\Sigma^{ns}(K)$.

    Suppose now that $m\equiv 2,3\bmod4$. In this case $s(\rho^m_m)=\left(\tfrac{-1,-1}{k}\right)$ (see \Cref{eq:2m hyp form}). Then by \Cref{lem: BF splitting hyperbolic}, condition {\it (iii)} is satisfied if and only if $\Ram\left(s(q)\cdot \left(\tfrac{-1,-1}{k}\right)\right)\subseteq\Sigma^{ns}(K)$.
    
    By \Cref{lem:ktoRk}, a conjugate $\gamma$ of $T$ is contained in $\Or(q')_{\Ok}$ for some $q'\simeq q$. Note that $\det(\gamma)=\det(T)=F(0)=1$, so $\gamma\in\SO(q')_{\Ok}$ and thus 
    \[ \ell(\gamma)=\ell(T)=\log\lambda \in\Q\mathcal{L}(\SO(q)_{\Ok}). \]    
\end{proof}

Next we consider the case when the difference is 1.

\begin{proposition} \label{prop: dif 1 conditions}
    Suppose $n+1-\deg(f)=1$ and $q$ is an admissible quadratic form over $k$ with signature $(n,1)$. 
    
    If $ \frac{n}{2}\equiv 0,1\bmod4$ then $\log\lambda\in\Q\mathcal{L}(\SO(q)_{\Ok})$ if and only if $\Ram\left(s(q)\right)\subset\Sigma^{ns}(K)$.
    
    If $ \frac{n}{2}\equiv 2,3\bmod4$ then $\log\lambda\in\Q\mathcal{L}(\SO(q)_{\Ok})$ if and only if $\Ram\left(s(q)\cdot \left(\tfrac{-1,-1}{k}\right)\right) \subseteq\Sigma^{ns}(K)$.
\end{proposition}

\begin{proof}
    Since the rank of $q$ is odd, up to scaling (and hence $k$-similarity), we may assume $\det(q)=-1$ in $\kk$. Let $m=\frac{n}{2}$. Up to $k$-similarity, it suffices by \Cref{cor:f vs fs} to consider the realization of $\lambda$ in $\SO(q)_{\Ok}$.
    
    Since $f(0)=1$ and $n+1-\deg(f)=1$, an isometry in $\SO(q)_k$ realizing $\lambda$ would need to have characteristic polynomial $F(x)=f(x)(x+1)$. By \Cref{Bf f0 deg 1}, there is an isometry in $\Or(q)_{\Ok}$ with characteristic polynomial $F$ if and only if $q\simeq\langle d_0\rangle\oplus q_1$ where $d_0=D\det(q)$ and $\Or(q_1)_{\Ok}$ has an isometry with characteristic polynomial $f$. Note that $\det(q_1)=D$ (up to squares). From the definition of the Hasse invariant as well as from elementary relations on Hilbert symbols we have 
    \[ s(q)= s(q_1)\cdot \left(\tfrac{\det(q_1),d_0}{k}\right) = s(q_1)\cdot \left(\tfrac{D,-D}{k}\right)=s(q_1). \]
    
    By \Cref{BF f0 0}, there is an isometry $T$ in $\Or(q_1)_{\Ok}$ with characteristic polynomial $f(x)$ if and only if the conditions {\it (i)-(iii)} are satisfied.  Condition {\it (i)} is immediately satisfied. Condition {\it (ii)} is satisfied by the proof of \Cref{lem:determinant neg}. 

    Suppose first that $\frac{n}{2}\equiv 0,1\bmod4$. In this case the Hasse invariant of the quadratic form $\rho^m_m$ has empty ramification (see \Cref{eq:2m hyp form}). Then by \Cref{lem: BF splitting hyperbolic}, condition {\it (iii)} is satisfied if and only if $\Ram\left(s(q)\right)=\Ram\left(s(q_1)\right)\subseteq\Sigma^{ns}(K)$.
      
    Suppose now that $\frac{n}{2}\equiv 2,3\bmod4$. In this case $s(\rho^m_m)=\left(\tfrac{-1,-1}{k}\right)$ (see \Cref{eq:2m hyp form}). Then by \Cref{lem: BF splitting hyperbolic}, condition {\it (iii)} is satisfied if and only if 
    \[ \Ram\left(s(q)\cdot \left(\tfrac{-1,-1}{k}\right)\right)  =\Ram\left(s(q_1)\cdot \left(\tfrac{-1,-1}{k}\right)\right)\subseteq\Sigma^{ns}(K) .\]
       
    By \Cref{lem:ktoRk}, a conjugate $\gamma$ of $T$ is contained in $\Or(q'_1)_{\Ok}$ for some $q'_1\simeq q_1$ and hence $\gamma\in \Or(q')_{\Ok}$ for some $q'\simeq q$. Note that $\det(\gamma)=\det(T)=F(0)=1$, so $\gamma\in\SO(q_1')_{\Ok}\subset \SO(q')_{\Ok}$ and thus 
    \[ \ell(\gamma)=\ell(T)=\log\lambda \in\Q\mathcal{L}(\SO(q)_{\Ok}). \]   
\end{proof}

The last remaining case is when the difference is 2. In the previous cases, it sufficed to consider polynomials of the form $F(x)=f(x)f_0(x)$ where $f_0$ was of type 0. When the difference is 2, we will also need to consider polynomials of the form $F(x)=f(x)f_1(x)$ where $f_1$ is a degree 2 symmetric polynomial which is irreducible over $k$ and whose roots and all of their conjugates are in the unit circle. In this case, $f_1$ would need to be the minimal polynomial over $k$ of some root of unity $\zeta$ for which $\zeta\not\in k$ and $(\zeta+\zeta^{-1})\in k$. 

\begin{proposition} \label{prop:dif 2 conditions}
    Suppose $n+1-\deg(f)=2$ and $q$ is an admissible quadratic form over $k$ with signature $(n,1)$.

    If $\frac{n-1}{2}\equiv 1,2\bmod4$ then $\log\lambda\in\Q\mathcal{L}(\SO(q)_{\Ok})$ if and only if either
    \begin{itemize}
        \item[--] for any $\nu\in\Ram\left(s(q)\right)$ such that $\det(q)\equiv(-1)^{m+1}$ in $k_\nu^\times/k_\nu^{\times 2}$, $\nu\in\Sigma^{ns}(K)$,
        \item[--] or, if there exists a root of unity $\zeta$ such that $(\zeta+\zeta^{-1})\in k$, $\det(q)(4-\left(\zeta+\zeta^{-1})^2\right)\cong D$ in $\kk$ and for any $\nu\in\Ram\left(s(q)\cdot \left(\tfrac{-1,-1}{k}\right) \right)$ if $\nu$ splits in $k(\zeta)/k$ then $\nu\in\Sigma^{ns}(K)$.
    \end{itemize}

    If $ \frac{n-1}{2}\equiv 0,3\bmod4$ then $\log\lambda\in\Q\mathcal{L}(\SO(q)_{\Ok})$ if and only if 
    either
    \begin{itemize}
        \item[--] for any $\nu\in\Ram\left(s(q)\cdot \left(\tfrac{-1,-1}{k}\right) \right)$ such that $\det(q)\equiv(-1)^{m+1}$ in $k_\nu^\times/k_\nu^{\times 2}$, $\nu\in\Sigma^{ns}(K)$,
        \item[--] or, if there exists a root of unity $\zeta$ such that $(\zeta+\zeta^{-1})\in k$, $\det(q)(4-\left(\zeta+\zeta^{-1})^2\right)\cong D$ in $\kk$ and for any $\nu\in\Ram\left(s(q) \right)$ if $\nu$ splits in $k(\zeta)/k$ then $\nu\in\Sigma^{ns}(K)$.
    \end{itemize}
\end{proposition}

\begin{proof}
    Up to $k$-similarity, it suffices by \Cref{cor:f vs fs} to consider the realization of $\lambda$ in $\SO(q)_{\Ok}$.
    If $T$ is an isometry realizing $\lambda$, then the characteristic polynomial $F(x)$ of $T$ must have form either $f(x)(x\pm1)^2$ or $f(x)(x^2-(\zeta+\zeta^{-1})x+1)$ for some root of unity $\zeta$ with $\zeta\not\in k$ and $(\zeta+\zeta^{-1})\in k$.

    Let $m=\frac{n-1}{2}$.


    Consider first the case that $F(x)=f(x)(x\pm1)^2=f(x)f_0(x)$. By \Cref{Bf f0 deg 2}, there is an isometry in $\Or(q)_{\Ok}$ with characteristic polynomial $F$ if and only if the conditions {\it (i)-(ii)} are satisfied. Condition {\it (i)} is satisfied by the proof of \Cref{lem:determinant neg}. 

    Over local fields such as $k_\nu$, we know the following (see e.g. \cite[VI.2]{Lambook}):
    \begin{itemize}
        \item equivalence classes of quadratic forms are uniquely determined by their rank, determinant up to squares, and Hasse invariant,
        \item  the maximal rank of an anisotropic quadratic form is 4  and there is a unique such form up to equivalence: $\langle1, -u,-\pi,u\pi\rangle$ where $\pi$ is a uniformizer of $k_\nu$ and $u$ is a nonsquare unit such that $k_\nu\left(\sqrt{u}\right)$ is the unique unramified quadratic extension of $k_\nu$,
        \item and $\left(\tfrac{u,\pi}{k_\nu}\right)$ is the unique quaternion division algebra up to isomorphism.
    \end{itemize}
    
    Therefore, by \Cref{lem: BF splitting hyperbolic} 
    condition {\it (ii)} is equivalent to: if $q_\nu\simeq \rho^{m-1}_{m-1}\oplus\langle1, -u,-\pi,u\pi\rangle$, then $\nu\in\Sigma^{ns}(K).$

    The Hasse invariant and determinant of the form $\langle1, -u,-\pi,u\pi\rangle$ over $k_\nu$ are given by
    \[ s(\rho^{m-1}_{m-1}\oplus\langle1, -u,-\pi,u\pi\rangle)_\nu =\begin{cases}
        \left(\tfrac{-1,-1}{k_\nu}\right)\cdot \left(\tfrac{u,\pi}{k_\nu}\right) & \text{if } m\equiv0,3\bmod 4
        \\ \left(\tfrac{u,\pi}{k_\nu}\right) & \text{if } m\equiv1,2\bmod 4
    \end{cases} \]
    and
    \[ \det(\rho^{m-1}_{m-1}\oplus\langle1, -u,-\pi,u\pi\rangle)_\nu =\begin{cases}
        1 & \text{if } m\equiv1,3\bmod 4
        \\ -1 & \text{if } m\equiv0,2\bmod 4 
    \end{cases} .\]

    If $m\equiv1,2\bmod 4$, condition {\it (ii)} is satisfied if and only if whenever $\nu\in\Ram\left(s(q)\right)$, if $\det(q)\equiv(-1)^{m+1}$ in $k_\nu^\times/k_\nu^{\times 2}$ then $\nu\in\Sigma^{ns}(K)$.
    If $m\equiv0,3\bmod 4$, condition {\it (ii)} is satisfied if and only if whenever $\nu\in\Ram\left(s(q) \cdot\left(\tfrac{-1,-1}{k}\right) \right)$, if $\det(q)\equiv(-1)^{m+1}$ in $k_\nu^\times/k_\nu^{\times 2}$ then $\nu\in\Sigma^{ns}(K)$.

    
    Now consider the case that $F(x)=f(x)f_1(x)$ where $f_1(x)=x^2-(\zeta+\zeta^{-1})x+1$ and $\zeta$ is a root of unity. The field $k(\zeta)$ must have degree 2 over $k$.
    Note first that unless $k=\Q$ and $[K:k]=2$, neither $f(x)$ nor $f_1(x)$ become hyperbolic over the real place $\nu_0$. Indeed, in this case $\lambda$ has a pair of non-real Galois conjugates $e^{i\theta},e^{-i\theta}$, and thus over $\nu_0$ $f_1(x)$ remains irreducible and $f(x)$ has a unique type 1 factor $x^2-(e^{i\theta}+e^{-i\theta})x+1$.
    If $k=\Q$ and $[K:k]=2$, then $K$ and $k(\zeta)$ are two distinct quadratic extensions of $\Q$ and there is a rational prime (in fact, infinitely many) which remains inert in both extensions. 
    Therefore, in all cases there exists a place $\nu$ of $k$ which is non-split in the splitting fields of $K$ and $k(\zeta)$. We can apply \cite[Corollary 12.6]{Bayer-Fluckinger} which says that the quadratic space $(V,q)$ has an isometry with characteristic polynomial $F$ if and only if the conditions {\it (i)-(iii)} in \Cref{BF f0 0} are satisfied.

    Condition {\it (i)} is satisfied if and only if $\det(q)(4-\left(\zeta+\zeta^{-1})^2\right)\equiv D$ in $\kk$.

    Condition {\it (ii)} is satisfied by the proof of \Cref{lem:determinant neg}.

     If $m\equiv 0,3\bmod 4$, the Hasse invariant of the quadratic form $\rho^{m+1}_{m+1}$ has empty ramification (see \Cref{eq:2m hyp form}). Then by \Cref{lem: BF splitting hyperbolic}, condition {\it (iii)} is satisfied if and only if for any $\nu \in\Ram\left(s(q)\right)$ either $\nu\in\Sigma^{ns}(K)$ or $\nu$ does not split in $k(\zeta)/k$.
     
     Suppose now that $m\equiv 1,2\bmod4$. In this case $s(\rho^{m+1}_{m+1})=\left(\tfrac{-1,-1}{k}\right)$ (see \Cref{eq:2m hyp form}). Then by \Cref{lem: BF splitting hyperbolic}, condition {\it (iii)} is satisfied if and only if for any $\nu\in\Ram\left(s(q)\cdot\left(\tfrac{-1,-1}{k}\right)\right)$ either $\nu\in\Sigma^{ns}(K)$ or $\nu$ does not split in $k(\zeta)/k$.
            
    By \Cref{lem:ktoRk}, a conjugate $\gamma$ of $T$ is contained in $\Or(q')_{\Ok}$ for some $q'\simeq q$. Note that $\det(\gamma)=\det(T)=F(0)=1$, so $\gamma\in\Gamma=\SO(q')_{\Ok}$ and thus 
    \[ \ell(\gamma)=\ell(T)=\log\lambda \in\Q\mathcal{L}(\SO(q)_{\Ok}). \] 
\end{proof}

\section{Infinitely many commensurability classes} \label{s:infcommclasses}

We are now able to construct infinitely many commensurability classes of arithmetic lattices realizing a Salem number $\lambda$.

We will need some additional notation.
\begin{itemize}
    \item $g(x)\in k[x]$ the minimal polynomial of $\lambda+\lambda^{-1}$ over $k$.
    \item $\delta=(-1)^{n(n+1)/2}D$ as an element in $\kk$.
    \item $h(x)=x^2-\delta\in k[x]$ the minimal polynomial of $\delta$ over $k$.
    \item $H=k(\sqrt{\delta})\cong k[x]/(h)$.
    \item $M=E'H$ the compositum of $E'$ and $H$.
    \item $L=K'M$ the compositum of $K'$ and $M$.
    \item For $r\in\N$, $Q_r=\langle1,\dots,1\rangle$ the standard quadratic form of signature $(r,0)$.
\end{itemize}
Note that the fields $H$, $M$, and $L$ are Galois extensions of $k$.

Over the splitting field (or over $\C$), $f$ factors as
\begin{equation} \label{eq:C factor}
    f(x)= (x-\lambda)(x-\lambda^{-1})(x-e^{i\theta_1})(x-e^{-i\theta_1})\cdots(x-e^{i\theta_{m-1}})(x-e^{-i\theta_{m-1}}),
\end{equation}
which equals
\begin{equation} \label{eq:trace factor}
    f= (x^2-(\lambda+\lambda^{-1})x+1)(x^2-2\cos\theta_1x+1)\cdots(x^2-2\cos\theta_{m-1}x+1) .
\end{equation}
Over the $\R$, $f$ factors as
\begin{equation} \label{eq:R factor}
    f= (x-\lambda)(x-\lambda^{-1})(x^2-2\cos\theta_1x+1)\cdots(x^2-2\cos\theta_{m-1}x+1) .
\end{equation}

As in \Cref{sec:criteria}, we will prove \Cref{main theorem} in several cases according to the difference $n+1-\deg(f)$ as well as the parity of $n$ in \Cref{prop:geq3 inf comm,prop:2 inf comm,prop:even inf comm,prop:odd 2 inf comm,prop:odd 0 inf comm}. As before, the case when $n+1-\deg(f)\geq 3$ is the easiest.

\begin{proposition} \label{prop:geq3 inf comm}
    If $n+1-\deg(f)\geq 3$ then every commensurability class of arithmetic lattice of simplest type in $\Isom(\H^n)$ defined over $k$ realizes $\lambda$.
\end{proposition}

\begin{proof}
    This follows immediately from \Cref{prop:dif3 conditions}.
\end{proof}

The cases with $n+1-\deg(f)\leq 2$ will be treated depending on the parity of $n$.

Whenever $M$ is a field extension of $k$, denote by $\mathrm{Spl}(M/k)$ the set of prime ideals in $k$ that split completely in $M$. Whenever $S$ is any set of prime ideals in $k$ (usually finite), $\mathrm{Spl}_S(M/k):=\mathrm{Spl}(M/k)-S$.
The following lemma is a well-known extension of a theorem originally due to Bauer. We include a proof using the Chebotarev density theorem \cite{Chebotarev}.

\begin{lemma} \label{Bauer}
    Let $L$ and $M$ be Galois extensions of $k$. Suppose $S$ is a finite set of prime ideals in $k$. Then $L\subseteq M$ if and only if $\mathrm{Spl}_S(M/k)\subseteq \mathrm{Spl}_S(L/k)$.
\end{lemma}

\begin{proof}
    One direction is immediate. To prove the other direction, note first that 
    \[ \mathrm{Spl}_S(LM/k) = \mathrm{Spl}_S(L/k)\cap \mathrm{Spl}_S(M/k). \]
    Hence, if $\mathrm{Spl}_S(M/k)\subseteq \mathrm{Spl}_S(L/k)$ then $\mathrm{Spl}_S(LM/k)= \mathrm{Spl}_S(M/k)$. However, by the Chebotarev density theorem, these sets have density $1/[LM:k]=1/[M:k]$ and thus $LM=M$ and $L\subseteq M$.
\end{proof}

\begin{lemma}\label{lem: K not in M}
    $K\not\subseteq M$ except possibly in the case that $\deg_k(\lambda)=2$ and $\delta=-D$.
\end{lemma}

\begin{proof}
    We first note that $E'$ is a real field (in fact, totally real) since the conjugates over $\Q$ of $\lambda+\lambda^{-1}$ are $2\cos\theta_j$ for $1\leq j\leq m-1$, and thus the conjugates over $k$ are some subset of those. By construction, $\lambda\not\in E'$. Then $E'(\lambda)$ is quadratic over $E'$.
    
    If $\delta$ is a square in $E'$, then $M=E'$ and it is clear that $K\not\subseteq M$.

    Now assume $\delta$ is not a square in $E'$. Then $M=E'(\sqrt{\delta})$ is quadratic over $E'$, as is $E'(\lambda)$. Assume that $K\subseteq M$, then $E'\subseteq E'(\lambda)\subseteq M$, which implies $E'(\lambda)=M$. 
 
    Suppose that $\deg_k(\lambda)=2$ and $\delta=D<0$. This implies that $k=E=E'$ and $f(x)=x^2-(\lambda+\lambda^{-1})x+1$. Then $K$ and $M$ are both of degree 2 over $E$, one of which is real and the other of which is not real, which is a contradiction.    

    Now suppose that $\deg_k(\lambda)\geq 4$. Then $\lambda$ has non-real conjugates over $k$. By construction $M$ is Galois over $k$. But $E'(\lambda)$ is not Galois over $k$ since it is real and hence does not contain any root of $f$ other than $\lambda$ and $\lambda^{-1}$. Again, this is a contradiction.
\end{proof}

We now consider the case of $n=2$, which gives a similar result to \cite[Corollary 12.2.10]{MRBook}.

\begin{proposition} \label{prop:2 inf comm}
    Let $n=2$. If $\deg(f)=2$ then there exist infinitely many commensurability classes of arithmetic lattice of simplest type in $\Isom(\H^2)$ defined over $k$ which realize $\lambda$.
\end{proposition}

\begin{proof}
    Under these hypothesis, $K=L$, $E=k$, $M=H$, and $\delta=-D$. If $K=M$ then $\Sigma^{ns}(K)-\mathrm{Spl}(H/k)=\Sigma^{ns}(K)$. Otherwise, $K$ and $M$ are distinct Galois extensions of $k$, both of degree two. It follows from \Cref{Bauer} and the Chebotarev density theorem that in either case we have $\#\left(\Sigma^{ns}(K)-\mathrm{Spl}(M/k)\right)=\infty$. 
    
    Choose a finite subset of even cardinality $\mathcal{A}$ satisfying
    \[ \mathcal{A} \subseteq \left(\Sigma^{ns}(K)-\mathrm{Spl}(M/k)\right) .  \]

    Whenever $\nu\in\mathcal{A}$, $\nu\not\in \mathrm{Spl}(k(\sqrt{\delta})/k)$ since $M=k(\sqrt{\delta})$. This means $\delta$ is a non-square at $k_\nu$. Therefore, by \cite[Theorem 71:19]{Omeara}, there exists $a\in k^\times$ such that the quaternion algebra $\left( \frac{a,\delta}{k} \right)$ is ramified exactly at $\mathcal{A}$. 
    Define the admissible quadratic form
    \[ q=\langle -\delta a, a, 1\rangle \]
    which has Hasse invariant $\left( \frac{a,-a\delta}{k} \right)\cong \left( \frac{a,\delta}{k} \right)$.

    By construction, $\Ram\left(s(q)\right)\subseteq\Sigma^{ns}(K)$. Thus by \Cref{prop: dif 1 conditions}, $\lambda$ is realized in the commensurability class of $\SO(q)_{\Ok}$. Varying over different choices of $\mathcal{A}$, we get infinitely many quadratic forms $q$'s with pairwise distinct Hasse invariant, and hence pairwise distinct Witt invariant. By \Cref{Maclachlan}, this gives infinitely many commensurability classes.
\end{proof}

\begin{proposition} \label{prop:even inf comm}
    Let $n\geq 4$ be even. If $n+1-\deg(f)=1$ then there exist infinitely many commensurability classes of arithmetic lattice of simplest type in $\Isom(\H^n)$ defined over $k$ which realize $\lambda$.
\end{proposition}

\begin{proof}
    We claim first that $\#\Sigma^{ns}(K)=\infty$.

    The extension $K/E$ is a Galois extension of degree 2. By the Chebotarev density theorem, the set of prime ideals $\wp$ in $E$ which remain inert in $K$ has density $1/2$. At most $[E:k]$-many such $\wp$ can lie over the same prime ideal $\mathfrak{p}=\wp\cap \Ok$ of $k$. Thus $\Sigma^{ns}(K)$ is an infinite set.
 
    Choose a finite subset of even cardinality $\mathcal{A}$ satisfying
    \[ \mathcal{A} \subseteq\Sigma^{ns}(K), \text{ if } \frac{n}{2}\equiv 0,1\bmod4  \]
    and 
    \[ \Ram_{fin}\left(\tfrac{-1,-1}{k}\right)\subseteq\mathcal{A} \subseteq\Sigma^{ns}(K)\cup \Ram_{fin}\left(\tfrac{-1,-1}{k}\right), \text{ if } \frac{n}{2}\equiv 2,3\bmod4.  \]

    Let $B_\mathcal{A}$ be the quaternion algebra with ramification set $\mathcal{A}$, which only contains finite places. Choose $a,b\in k^\times$ so that
    \[ \left(\tfrac{a,b}{k}\right)=B_\mathcal{A}\cdot \left(\tfrac{-1,-D}{k}\right). \]
    Note that $\left(\tfrac{a,b}{k}\right)$ is ramified at all real places except $\nu_0$. By \Cref{lem:determinant neg}, at $\nu_0$ we have $D<0$ and at least one of $a$ or $b$ is positive. At any other real place $\nu$, $D>0$, $a<0$, and $b<0$.
    Define the admissible quadratic form 
    \[ q= \langle -Da,-Db,Dab \rangle \oplus Q_{n-2} \]
    which has Hasse invariant
    \[ s(q)=s\left(\langle -Da,-Db,Dab \rangle\right)
    =\left(\tfrac{a,b}{k}\right) \cdot \left(\tfrac{-1,-D}{k}\right) =B_\mathcal{A}
    .\]
    Note that $q$ has the correct signatures and is admissible. 

    By construction, if $\frac{n}{2}\equiv0,1\bmod4$, $\Ram\left(s(q)\right)\subseteq\mathcal{A}\subseteq \Sigma^{ns}(K)$. 

    If $\frac{n}{2}\equiv2,3\bmod4$, then by construction, $\Ram_{fin}\left(\left(\tfrac{-1,-1}{k}\right)\right) \subseteq \Ram\left(s(q)\right)=\mathcal{A}$  $\subseteq\Sigma^{ns}(K)$, so if $\nu\in\Ram\left(s(q)\cdot \left(\tfrac{-1,-1}{k}\right) \right)$ is a finite place, then $\nu\in\Sigma^{ns}(K)$.
    If $\nu$ is a real place then $f$ factors over $k_\nu$ into a subset of the factors in \Cref{eq:R factor}, with at least one factor of the type $(x^2-2\cos\theta_ix+1)$ (not repeated), since $\deg_k(f)\geq 4$. So $f$ is not hyperbolic at $k_\nu$ and hence by \Cref{lem: BF splitting hyperbolic}, $\nu\in\Sigma^{ns}(K)$.

    Thus by \Cref{prop: dif 1 conditions}, $\lambda$ is realized in the commensurability class of $\SO(q)_{\Ok}$.
    
    Varying over different choices of $\mathcal{A}$, we get infinitely many quadratic forms $q$'s with pairwise distinct Hasse and Witt invariants, and hence pairwise distinct Witt invariant. By \Cref{Maclachlan}, this gives infinitely many commensurability classes.
\end{proof}

\begin{proposition} \label{prop:odd 2 inf comm}
    Let $n\geq3$ be odd. If $n+1-\deg(f)=2$ then there exist infinitely many commensurability classes of arithmetic lattices of simplest type in $\Isom(\H^n)$ defined over $k$ which realize $\lambda$.
\end{proposition}

\begin{proof}
    As in the beginning of the proof of \Cref{prop:2 inf comm}, it follows from \Cref{Bauer,lem: K not in M} that 
    \[ \#\left(\Sigma^{ns}(K)\cap \mathrm{Spl}(H/k)\right)=\infty. \]
    Choose $\mathcal{A}$ to be a finite subset of even cardinality such that if $ \frac{n-1}{2}\equiv 0,1\bmod4$, then
    \[ \mathcal{A}\subseteq\Sigma^{ns}(K)\cap \mathrm{Spl}(H/k)  \]
    and if $\frac{n-1}{2}\equiv 2,3\bmod4$, then
    \[ \Ram_{fin}\left(\left(\tfrac{-1,-1}{k}\right)\right) \subseteq \mathcal{A}\subseteq \left(\Sigma^{ns}(K)\cap \mathrm{Spl}(H/k)\right)\cup \Ram_{fin}\left(\left(\tfrac{-1,-1}{k}\right)\right).  \]

    Let $B_\mathcal{A}$ be the quaternion algebra with ramification set $\mathcal{A}$, which only contains finite places. Choose $a,b\in k^\times$ so that
    \[ \left(\tfrac{a,b}{k}\right)=B_\mathcal{A}\cdot \left(\tfrac{-1,-D}{k}\right). \]
    Note that $\left(\tfrac{a,b}{k}\right)$ is ramified at all real places except $\nu_0$. By \Cref{lem:determinant neg}, at $\nu_0$ we have $D<0$ and at least one of $a$ or $b$ is positive. At any other real place $\nu$, $D>0$, $a<0$, and $b<0$.
    Define the admissible quadratic form
    \[ q=\langle -Da,-Db,Dab\rangle \oplus Q_{n-2} \]
    which has Hasse invariant
    \[ s(q)=s\left(\langle -Da,-Db,Dab \rangle\right)
    =\left(\tfrac{a,b}{k}\right) \cdot \left(\tfrac{-1,-D}{k}\right)=B_\mathcal{A}
    .\]
    Note that $q$ has the correct signatures and is admissible. 

    If $\frac{n-1}{2}\equiv 0,1\bmod4$, $\Ram\left(s(q)\right)=\mathcal{A}\subseteq\Sigma^{ns}(K)$. 
   
    If $\frac{n-1}{2}\equiv2,3\bmod4$, then by construction, $\Ram_{fin}\left(\left(\tfrac{-1,-1}{k}\right)\right) \subseteq \Ram\left(s(q)\right)=\mathcal{A}$  $\subseteq\Sigma^{ns}(K)$, so if $\nu\in\Ram\left(s(q)\cdot \left(\tfrac{-1,-1}{k}\right) \right)$ is a finite place, then $\nu\in\Sigma^{ns}(K)$.
    If $\nu$ is a real place then $f$ factors over $k_\nu$ into a subset of the factors in \Cref{eq:R factor}, with at least one factor of the type $(x^2-2\cos\theta_ix+1)$ (not repeated), since in this case $n\geq5$ and $\deg(f)\geq 4$. So $f$ is not hyperbolic at $k_\nu$ and hence by \Cref{lem: BF splitting hyperbolic}, $\nu\in\Sigma^{ns}(K)$.

    Thus by \Cref{prop:dif 2 conditions}, $\lambda$ is realized in the commensurability class of $\SO(q)_{\Ok}$.
       
    Varying over different choices of $\mathcal{A}$, we get infinitely many quadratic forms $q$'s with pairwise distinct Hasse and Witt invariants. Furthermore, if $\mathfrak{p}\in \mathcal{A}$ but $\mathfrak{p}\not\in \mathcal{A}'$ for two such $q$ and $q'$ constructed as above, then since $\mathfrak{p}\in \mathrm{Spl}(H/k)$, by \Cref{lem:Mac odd non comm} $\SO(q)_{\mathcal{O}_k}$ and $\SO(q')_{\mathcal{O}_k}$ are incommensurable. 
    
    Again, this gives infinitely many commensurability classes. 
\end{proof}

\begin{proposition} \label{prop:odd 0 inf comm}
    Let $n\geq3$ be odd. If $n+1=\deg(f)$ then there exist infinitely many commensurability classes of arithmetic lattice of simplest type in $\Isom(\H^n)$ defined over $k$ which realize $\lambda$.
\end{proposition}

\begin{proof} 
    It follows from \Cref{Bauer,lem: K not in M} that
    \[ \#\left(\Sigma^{ns}(K)\cap \mathrm{Spl}(H/k)\right)=\infty. \]
    Choose $\mathcal{A}$ to be a finite subset of even cardinality such that if $ \frac{n+1}{2}\equiv 0,1\bmod4$, then
    \[ \mathcal{A}\subseteq \Sigma^{ns}(K)\cap \mathrm{Spl}(H/k)  \]
    and if $\frac{n+1}{2}\equiv 2,3\bmod4$, then
    \[ \Ram_{fin}\left(\tfrac{-1,-1}{k}\right)\subseteq \mathcal{A}\subseteq\left(\Sigma^{ns}(K)\cap \mathrm{Spl}(H/k)\right)\cup \Ram_{fin}\left(\tfrac{-1,-1}{k}\right).  \]
    
    Let $B_\mathcal{A}$ be the quaternion algebra with ramification set $\mathcal{A}$, which only contains finite places. Choose $a,b\in k^\times$ so that
    \[ \left(\tfrac{a,b}{k}\right)=B_\mathcal{A}\cdot \left(\tfrac{-1,-D}{k}\right). \]
    Note that $\left(\tfrac{a,b}{k}\right)$ is ramified at all real places except $\nu_0$. By \Cref{lem:determinant neg}, at $\nu_0$ we have $D<0$ and at least one of $a$ or $b$ is positive. At any other real place $\nu$, $D>0$, $a<0$, and $b<0$.
    Define the admissible quadratic form
    \[ q=\langle -Da,-Db,Dab\rangle \oplus Q_{n-2} \]
    which has Hasse invariant
    \[ s(q)=s\left(\langle -Da,-Db,Dab \rangle\right)
    =\left(\tfrac{a,b}{k}\right) \cdot \left(\tfrac{-1,-D}{k}\right)=B_\mathcal{A}
    .\]
    Note that $q$ has the correct signatures and is admissible.

    If $\frac{n+1}{2}\equiv 0,1\bmod4$, $\Ram\left(s(q)\right)=\mathcal{A}\subseteq\Sigma^{ns}(K)$. 
   
    If $\frac{n+1}{2}\equiv2,3\bmod4$, then by construction, $\Ram_{fin}\left(\left(\tfrac{-1,-1}{k}\right)\right) \subseteq \Ram\left(s(q)\right)=\mathcal{A}$  $\subseteq\Sigma^{ns}(K)$, so if $\nu\in\Ram\left(s(q)\cdot \left(\tfrac{-1,-1}{k}\right) \right)$ is a finite place, then $\nu\in\Sigma^{ns}(K)$.
    If $\nu$ is a real place then $f$ factors over $k_\nu$ into a subset of the factors in \Cref{eq:R factor}, with at least one factor of the type $(x^2-2\cos\theta_ix+1)$ (not repeated), since in this case $n\geq3$ and $\deg(f)\geq 4$. So $f$ is not hyperbolic at $k\nu$ and hence by \Cref{lem: BF splitting hyperbolic}, $\nu\in\Sigma^{ns}(K)$.

    Thus by \Cref{prop:dif 0 conditions}, $\lambda$ is realized in the commensurability class of $\SO(q)_{\Ok}$.
       
    Varying over different choices of $\mathcal{A}$, we get infinitely many quadratic forms $q$'s with pairwise distinct Hasse and Witt invariants. Furthermore, if $\mathfrak{p}\in \mathcal{A}$ but $\mathfrak{p}\not\in \mathcal{A}'$ for two such $q_1$ and $q_2$ constructed as above, then since $\mathfrak{p}\in \mathrm{Spl}(H/k)$, by \Cref{lem:Mac odd non comm} $\SO(q_1)_{\mathcal{O}_k}$ and $\SO(q_2)_{\mathcal{O}_k}$ are incommensurable. 
\end{proof}

\bibliographystyle{alpha}
\bibliography{biblio}

\end{document}